\documentclass[10pt,journal,final,twoside,letterpaper]{IEEEtran}
\bstctlcite{IEEEexample:BSTcontrol}

\usepackage{setspace}
\usepackage{enumitem}
\usepackage{hyperref}
\hypersetup{
 colorlinks =true,
 linkcolor = blue,
 filecolor = blue,
 citecolor = magenta, 
 urlcolor = magenta,
 }
\usepackage{cite,graphicx}
\usepackage{amsmath,bm,amssymb,amsthm}
\usepackage{booktabs,multicol,multirow}

\theoremstyle{remark}

\newtheorem{definition}{Definition}

\newtheorem{proposition}{Proposition}

\newtheorem{corollary}{Corollary}

\begin{document}
\title{Fairly Extreme: Minimizing Outages Equitably}
\author{Kaarthik Sundar$^\dag$, Deepjyoti Deka$^\dag$ and Russell Bent$^\dag$
\thanks{\noindent 
$^\dag$Los Alamos National Laboratory, New Mexico, USA
}\;
\thanks{\noindent E-mail ids: \texttt{\{kaarthik,deepjyoti,rbent\}@lanl.gov}}\;
\thanks{\noindent The authors acknowledge the funding provided by LANL’s Directed Research and Development project: ``Resilient operation of interdependent engineered networks and natural systems'' and DOE's Grid Modernization Initiative project: ``Climate-Resilient Equitable Resource Planning''. The research work conducted at Los Alamos National Laboratory is done under the auspices of the National Nuclear Security Administration of the U.S. Department of Energy under Contract No. 89233218CNA000001.}\;
}

\maketitle

\begin{abstract}
This paper focuses on the problem of minimizing the outages due to extreme events on the power grid equitably among all customers of the grid. The paper presents two ways of incorporating fairness into the existing formulations that seek to minimize the total outage in the power grid. The first method is motivated by existing literature on incorporating fairness in optimization problems and this is done by modifying the problem's objective function. The second method introduces a novel notion of fairness, termed $\varepsilon$-fairness, that can be incorporated into existing problem formulations through a single second-order cone constraint. Both these methods are very general and can be used to incorporate fairness in existing planning and operational optimization problems in the power grid and beyond. Extensive computational case studies that examine the effectiveness of both these methods to characterize fairness is presented followed by conclusions and ways forward.
\end{abstract}

\begin{IEEEkeywords}
fairness; extreme events; power systems; optimization; equitable
\end{IEEEkeywords}

\section{Introduction} \label{sec:intro}
In the face of ever-increasing, climate-driven extreme events, power system operators need to manage their systems in a way as to minimize the total impact any event can have on all customers. In particular, the total impact is almost always measured as the sum total of the power outage on individual customers. Such an operating point that minimizes the sum of outages (or load shed), though cost-effective in terms of the total outage, may be ``unfair''. In other words, a majority of the burden of outages may fall on a subset of the customers. Nevertheless, it is also not difficult to realize that any attempt to be fair to all customers in terms of discomfort/outage will necessarily come at the expense of increasing the overall impact, i.e., sum
of outages, of the extreme event to the system. Thus, there inherently is a trade-off between cost and fairness that the operator must make to ensure outages are minimized equitably among all customers. This paper seeks to mathematically elucidate this trade-off by developing two concrete models of incorporating fairness into existing power system optimization problem formulations that seek to manage impacts from extreme events. 

Modeling and incorporating fairness in optimization problems is not new, and its effect on the solutions of the optimization problem has been studied in a variety of settings ranging from communication networks \cite{bertsekas2021data} and financial applications \cite{stubbs2009multi} to general combinatorial optimization problems (see \cite{bektacs2020using} and references therein). Furthermore, different metrics to measure fairness \cite{Gini1936,Jain1984} have been widely used across different domains. The importance of incorporating fairness into power system operations in an evolving power grid has come to the fore as both researchers and system operators are becoming cognizant of the social and health impacts that operations can have on the society and energy justice considerations (see \cite{Mathieu2023,taylor2023managing,huang2023inequalities,kody2022sharing,wang2023local}). This work takes a step in that direction by developing rigorous mathematics and, as a result, a simple framework to incorporate the notion of fairness into operational optimization problems that exist the power system literature. In particular, we focus on answering the following two questions in a power system optimization context where the objective function of most, if not all, of the optimization problems focus on decreasing the total cost of operations in one way or another: 
\begin{enumerate}[label=(Q\arabic*)]
 \item Are there improved objective functions that can be used to explore the trade-off between fairness and cost of operations? 
 \item Can we formulate fairness as a convex constraint, include it to existing power system optimization problems and explore the aforementioned trade-off in a systematic manner?
\end{enumerate}
We answer both these questions in the affirmative and compare and contrast both these approaches using the Minimum Load Shedding (MLS) problem \cite{coffrin2018relaxations} with a Direct Current (DC) power flow model that focuses on minimizing the sum of the load shed to each customer after an extreme event damages a subset of components of the power system. Though all our models generalize to the the full Alternating Current (AC) power flow equations, we restrict our focus on the DC model for ease of exposition. 

The reason behind posing the problem of incorporating fairness in power system optimization in two different ways, one as an objective (Q1) and another as a convex constraint (Q2), is that extensive work in operations research for (Q1) exists but has not sufficiently applied to power system research, while to the best of our knowledge, none exits for (Q2). So the dual focus of this article is to bridge the gap between existing research for (Q1) and critical problems in the power community, and to further develop a novel technique to answer (Q2), that is applicable above and beyond power systems. In the process of doing so, we shall demonstrate that the novel technique developed to answer (Q2) is more flexible and offers a better approach to incorporate fairness into existing grid optimization problems. Furthermore, we also remark that the approaches we develop are very general and can be suitably modified to work for capacity expansion problems as well. 

Research to answer (Q1) exists in the operations research literature for applications like vehicle routing, facility location, and nurse rostering (see \cite{bektacs2020using,barbati2016equality, matl2018workload,martin2013cooperative} and references therein), to name a few. In the power systems literature, equity and load prioritization has been used interchangeably; existing work in the power literature enforces equity (prioritizing disadvantageous users) through the surrogate of minimizing the weighted load shed where weights denote the importance/priority of the customers that constitute that load \cite{taylor2023managing,wang2023local}. This can, however, lead to unfair solutions when fairness is measured using metrics and also because weights can only model priorities when performing load shed rather than fairness \cite{Bertsimas2011}. More recently, authors in \cite{kody2022sharing} use min-max load shed to model fairness and a trade-off parameter to combine it with total load shed. For (Q1), we explore the use of alternate objective functions based on $\ell^p$ norms \cite{bektacs2020using} with varying $p$ values to model fairness in the MLS problem. Note that this is a generalization of the work in \cite{kody2022sharing} as the min-max cost corresponds to $p =\infty$. We also remark that a notion of fairness referred to as ``$\alpha$-fairness'' exists in the literature when the optimization problem is formulated as a utility maximization problem \cite{mo2000fair}; it is not clear how one would formulate the MLS problem in such a form and hence this approach is not pursued to tackle (Q1) in this work.

As for (Q2), we develop a novel non-trivial Second-Order Cone (SOC) constraint with a scalar parameter to model fairness, with the parameter serving to measure the minimum level of fairness that an operator warrants in the optimal solution. Prior efforts to incorporate fairness as constraints in the context of the MLS include using bounds, chosen heuristically, on the maximum load shed at each load \cite{zukerman2008fair}. This can, however, lead to the optimization problem becoming infeasibile due to improper choice of bounds. 
In contrast, by directly encoding a fairness measure as a convex constraint we seek to avoid computing trade-off weights or variable bounds to indirectly enhance fairness. For MLS, our convex formulation directly finds the optimal load shed that has a minimum level of fairness encoded using the parameter. It can be immediately extended to include weighted load shedding cost as well, as necessitated for internalizing priority consideration \cite{taylor2023managing}. Furthermore, we demonstrate, theoretically and through simulations, that by changing the scalar parameter in our proposed method, one can achieve a monotonic increase in the commonly used fairness metric in \cite{Jain1984}.

\section{Mathematical Preliminaries} \label{sec:prelim}
This section formally introduces the (i) literature's MLS problem (ii) fairness indices used for measuring fairness in the load shed across customers obtained by solving different fair variants of the MLS problem and (iii) metrics used to quantify the trade-off between fairness and cost of operations in the MLS problem. We start by presenting the mathematical formulation for the MLS problem. 

\subsection{Minimum Load Shedding (MLS) Problem} \label{subsec:mls} 
Given a damaged power network with undamaged transmission lines $\mathcal L$, buses $\mathcal B$, generators $\mathcal G$ and loads $\mathcal D$, with decision variables $\theta_b$ that denotes the phase angle at bus $b \in \mathcal B$, $p_g$ that denotes the real power generated by generator $g \in \mathcal G$ and non-negative variable $d_i$ that denotes the amount of load shed at $i \in \mathcal D$ and parameters $p_{ij}^{\max}, b_{ij} ~\forall (i, j) \in \mathcal L$ (thermal limit and susceptance of line), $p_g^{\min}, p_g^{\max} ~\forall g \in \mathcal G$ (generation limits), $d_i^{\max} ~\forall i \in \mathcal D$ (active demands), the MLS problem is formulated as:
\begin{equation*}
 \min \sum_{i \in \mathcal D} d_i \text{ s.t.: } \mathcal X_{\mathrm{DC}} \triangleq \left\{ 
 \begin{gathered}
 \sum_{g \in \mathcal G} p_g = \sum_{i \in \mathcal D} (d_i^{\max} - d_i) \\ 
 p_g \in [p_g^{\min}, p_g^{\max}], d_i \in [0, d_i^{\max}] \\ 
 -p_{ij}^{\max} \leqslant b_{ij}(\theta_i - \theta_j) \leqslant p_{ij}^{\max}
 \end{gathered}
 \right\} % \label{eq:mls-full}
\end{equation*}

If we let $\bm d \in \mathbb{R}_{\geqslant 0}$ denote the vector of the individual load sheds $d_i$, $i \in \mathcal D$ then the MLS problem can be can also be represented concisely (with some abuse of notation) as 
\begin{gather}
 \bm v^* \triangleq \operatornamewithlimits{argmin}_{\bm d \in \mathcal X_{\mathrm{DC}}} \sum_{i \in \mathcal D} d_i = \operatornamewithlimits{argmin}_{\bm d \in \mathcal X_{\mathrm{DC}}} \| \bm d \|_1 \label{eq:mls}
\end{gather}
In \eqref{eq:mls}, we use the $\ell^1$-norm, $\| \cdot \|_1$, to denote the total load shed in the system as the individual load shed $d_i \geqslant 0$ for every $i \in \mathcal D$. 

When non-negative weights $w_i \geqslant 0$ for each $i\in \mathcal D$, that denotes the importance or priority of the load $i$, are provided, the weighted version of the MLS can be formulated as 
\begin{gather}
 \bm {v}_{\bm w}^* \triangleq \operatornamewithlimits{argmin}_{\bm d \in \mathcal X_{\mathrm{DC}}} \sum_{i \in \mathcal D} w_i \cdot d_i = \operatornamewithlimits{argmin}_{ \bm d \in \mathcal X_{\mathrm{DC}}} \| \bm w \circ \bm d \|_1 \label{eq:weighted-mls}
\end{gather}
where $\bm w$ denotes the vector of weights and $\bm x\circ \bm y$ denotes component-wise product of vectors $\bm x$ and $\bm y$.

We remark that some variant of \eqref{eq:mls} and \eqref{eq:weighted-mls} is routinely used by power system operators to restore the system to its normalcy \cite{rhodes2021powermodelsrestoration,coffrin2018relaxations} or even to perform public-service power shutoffs \cite{rhodes2020balancing}. Next, we present two important metrics that are widely accepted as measures of fairness or unfairness among different components of any vector $\bm x \in \mathbb R_{\geqslant 0}^n$

\subsection{Fairness Indices} \label{subsec:fairness-indices} 
Gini \cite{Gini1936} proposed to measure the unfairness of a given vector $\bm x \in \mathbb R_{\geqslant 0}^n$
by computing the following ``index'' (also refered to as Gini Index): 
\begin{gather}
 \text{Gini Index: } \mathrm{GI}(\bm x) \triangleq \frac{\sum_{1 \leqslant i \leqslant j \leqslant n} | x_i - x_j |}{(n-1) \cdot \| \bm x \|_1}. \label{eq:gi}
\end{gather}
$\mathrm{GI}$ takes the value $1$ when the vector is most unfair i.e., exactly only element of the vector is non-zero and it takes the value $0$ when all the elements of the vector are equal.

The next measure is by Jain et. al \cite{Jain1984} who proposed to measure the fairness of $\bm x \in \mathbb R_{\geqslant 0}^n$ by computing the following index:
\begin{gather}
 \text{Jain et. al Index: } \mathrm{JI}(\bm x) \triangleq \frac 1n \cdot \frac{\| \bm x \|_1^2}{\| \bm x \circ \bm x \|_1}. \label{eq:ji}
\end{gather}
$\mathrm{JI}$ takes the value $1/n$ when exactly only element of the vector is non-zero and $1$ when all the elements of the vector are equal. Note that the denominator in \eqref{eq:ji} is the square of the 2-norm of $\bm x$. Throughout the rest of the article, we shall use these two indices to measure how fair or unfair the load shed or the weighted load shed provided by the MLS is to all the loads in the network.

\subsection{Price of Fairness} \label{subsec:pof} 
Price of fairness (POF) is a metric that is used in the literature \cite{Bertsimas2011} to study the trade-off between cost and fairness for resource allocation problems. Here, we use it to study the cost-fairness trade-off for different models of fairness included in the MLS problem. To define POF, let $F$ denote any version of the MLS problem with some model of fairness (incorporated either using (Q1) or (Q2)) in \eqref{eq:mls} and $\bm z^*(F)$ denote the vector of load sheds that is obtained by solving $F$. Then the POF is defined as:
\begin{gather}
 \mathrm{POF}(F) \triangleq \frac{\|\bm z^*(F)\|_1 - \|\bm v^*\|_1}{\|\bm v^*\|_1}. \label{eq:pof}
\end{gather}
The POF for $F$ in \eqref{eq:pof} is the relative increase in the total load shed under the fair solution $\bm z^*(F)$, compared to the MLS
solution of \eqref{eq:mls}, $\bm v^*$. In \eqref{eq:pof}, $\| \cdot \|_1$ is used to denote the total load shed as it equals the summation of all the individual components of the load shed vector. Note that POF is $\geqslant 0$
because the $\bm z^*$ is the minimum load shed that can be achieved by any fair MLS problem. The values of $\mathrm{POF}(F)$ closer to $0$ are preferable for any $F$ because it indicates less compromise on cost of operations while enforcing fairness. The definition \eqref{eq:pof} can be extended to the weighted version of the MLS problem in \eqref{eq:weighted-mls} by using $\bm v_{\bm w}^*$ in place of $\bm v^*$. 
The scope of this work is to quantify POF when using different models of fairness with the MLS problem. 

\section{Incorporating Fairness in the Objective} \label{sec:q1}
In this section, we answer (Q1) by relying on recent literature on incorporating fairness in combinatorial optimization problems \cite{bektacs2020using}. We propose to use $\ell^p$ norms with $p > 1$ instead of the $\|\cdot \|_1$ norm in the objective function of the MLS problems in \eqref{eq:mls} to obtain various fair versions of the same. We start by presenting additional mathematical notation. 

First, let $\ell^p$ be the norm of any vector $\bm x \in \mathbb R^n$ is a function whose domain and range are $\mathbb R^n$ and $\mathbb R$, respectively, and is defined as 
\begin{gather}
 \|x\|_p \triangleq \left( \sum_{i = 1}^n |x_i|^p\right)^{\frac 1p} \label{eq:lpnorm}.
\end{gather}
When $p = \infty$, the definition of $\ell^{\infty}$ norm is 
\begin{gather}
 \|x \|_{\infty} \triangleq \max_{i} |x_i| \label{eq:linfnorm}.
\end{gather}
For ease of exposition, we shall restrict $p$ to be integer, although our presentation holds any real $p$ as well. Also, it is easy to see that for any $p \geqslant 1$, the $\ell^p$ norm is a convex function and hence minimizing the $\ell^p$ norm of any vector in its domain results in a convex optimization problem. Furthermore, since raising to the power $p \in (1, \infty)$ is a monotonic function for non-negative reals, minimizing the $\| \bm x \|_p$ is equivalent to minimizing $\| \bm x\|_p^p$ \cite{kostreva2004equitable}, but for the $\ell^{\infty}$ norm. For the $\ell^{\infty}$, the objective can directly be reformulated to a linear objective. 

Given these notations, we let $F^o_p$ denote the fair version MLS in \eqref{eq:mls}. Here the superscript ``$o$'' represents the fact that fairness is being incorporated using the objective function and subscript ``$p$'' stands for the $p$-norm that is being used in the objective function. Now, the formulations for $F^o_p$ is as follows:
\begin{gather}
 \bm z^*(F_p^o) \triangleq \operatornamewithlimits{argmin}_{\bm d \in \mathcal X_{\mathrm{DC}}} \sum_{i \in \mathcal D} d_i^p = \operatornamewithlimits{argmin}_{\bm d \in \mathcal X_{\mathrm{DC}}} \| \bm d \|_p^p \label{eq:p-norm-mls}
\end{gather}
When $p = \infty$, the formulations correspond to
\begin{gather}
 \bm z^*(F_{\infty}^o) \triangleq \operatornamewithlimits{argmin}_{\bm d \in \mathcal X_{\mathrm{DC}}} \left( \max_{i \in \mathcal D} d_i \right) = \operatornamewithlimits{argmin}_{\bm d \in \mathcal X_{\mathrm{DC}}} \| \bm d \|_{\infty} \label{eq:p-norm-mls-inf}
\end{gather}
Though it is an open question to understand why changing the objective from a $\ell^1$ norm to an $\ell^p$ norm captures fairness for $p \in (1, \infty)$, informal explanations as to why this is the case can be found in \cite{bektacs2020using,kostreva2004equitable,gupta2023lp}. As for the case $p = \infty$, it is known in economics and game-theory that this notion of fairness is a generalization of the
Rawlsian justice \cite{Rawls1971} and the Kalai-Smorodinsky solution of a two-player game for multiple players \cite{kalai1975other}, or in the MLS case, multiple loads. 

As mentioned in the beginning of this section, \eqref{eq:p-norm-mls} is a convex optimization problem and \eqref{eq:p-norm-mls-inf} can be reformulated to a linear optimization problem. The reformulation for \eqref{eq:p-norm-mls-inf} is shown below:
\begin{gather}
 \bm z^*(F_{\infty}^{o}) = \operatornamewithlimits{argmin}~ \{ y : y \geqslant d_i ~\forall i \in \mathcal D, \bm d \in \mathcal X_{\mathrm{DC}} \}. \label{eq:linear-reformulation-min-max}
\end{gather}

Later in the case studies, we examine the impact of the various $\ell^p$ norms in the objective function on capturing the cost-fairness trade-off using the fairness indices and the price of fairness, introduced in Sec. \ref{sec:prelim}.

\section{Fairness as a Second-Order Cone Constraint} \label{sec:q2}
We now focus on (Q2) by developing a novel parameterized, SOC characterization of fairness. To that end, we start by invoking a well-known inequality that relates $\|\bm x\|_1$ and $\|\bm x\|_2$ for any $\bm x \in \mathbb R^n$, derived using the fundamental result of ``equivalence of norms'' (see \cite{horn2012matrix})
\begin{gather}
 \|\bm x \|_2 \leqslant \| \bm x\|_1 \leqslant \sqrt{n} \cdot \|\bm x\|_2. \label{eq:norm-eq}
\end{gather} 
In the above inequality, it is not difficult see that 
\begin{enumerate}[label=(\roman*)]
 \item the first inequality in \eqref{eq:norm-eq} holds at equality, i.e., $\|\bm x \|_2 = \| \bm x\|_1$, when $x_i$ is zero for all but one component (\textbf{most unfair}) and
 \item the second inequality in \eqref{eq:norm-eq} holds at equality, i.e., $\| \bm x\|_1 = \sqrt{n} \cdot \|\bm x\|_2$, when every component of $\bm x$ is equal to one another (\textbf{most fair}).
\end{enumerate}
Note that any $\bm x \in \mathbb R^n$ that satisfies (i) or (ii) also have Gini Index (see \eqref{eq:gi}) value $0$ or $1$, respectively. Furthermore, these conditions directly encode the bounds on the Jain et. al Index (see \eqref{eq:ji}), i.e., when $\bm x$ satisfies (i) or (ii), the value of its Jain et. al Index is $1/n$ or $1$, respectively. We use this intuition to introduce the following parameterized definition of fairness: 

\begin{definition}
\label{def:eps-fairness} \textit{$\varepsilon$-fairness} -- 
For any $\bm x \in \mathbb R^n$ and $\varepsilon \in [0, 1]$, we say $\bm x$ is ``$\varepsilon$-fair'', if $\left(1-\varepsilon + \varepsilon\sqrt{n} \right) \cdot \|\bm x \|_2 = \|\bm x\|_1$. Here, we say $\bm x$ is most unfair (fair) when $\varepsilon = 0$ ($\varepsilon = 1$), respectively. Indeed, at that value, $\bm x$ attains the lower (upper) bound in \eqref{eq:norm-eq}. Furthermore, we define $\bm x$ to be \textbf{``at least $\varepsilon$-fair''}, if $\left(1-\varepsilon + \varepsilon\sqrt{n} \right) \cdot \|\bm x \|_2 \leqslant \|\bm x\|_1$.
\end{definition}

We remark that $\left(1-\varepsilon + \varepsilon\sqrt{n} \right)$ in definition \ref{def:eps-fairness} is the convex combination of $1$ and $\sqrt n$, with $\varepsilon$ as the multiplier. The definition also allows modification of the MLS problem in \eqref{eq:mls} to explore the set feasible load sheds that are at least $\varepsilon$-fair by addition of the following SOC constraint 
\begin{gather}
 \left(1-\varepsilon + \varepsilon\sqrt{|\mathcal D|} \right) \cdot \| \bm d \|_2 \leqslant \| \bm d \|_1 \quad \forall \bm d \in \mathcal X_{\mathrm{dc}}. \label{eq:soc-e-fair} 
\end{gather}
When $\varepsilon$ is set to $0$ the \eqref{eq:soc-e-fair} enforces the trivial bound on the total load shed in the system, i.e., $\| \bm d \|_2 \leqslant \| \bm d \|_1$. Armed with these notations and \eqref{eq:soc-e-fair}, we are now ready to formulate a fair version of the MLS as 
\begin{gather}
 \bm z^*(F_{\varepsilon}^c) \triangleq \operatornamewithlimits{argmin} \| \bm d \|_1 \text{ subject to: \eqref{eq:soc-e-fair}.} \label{eq:fair-mls} 
\end{gather}
Here, the superscript ``$c$'' in $F_{\varepsilon}^c$ is used to denote the fact that fairness is enforced in a constraint form. 
We now present some desirable properties of $F_{\varepsilon}^c$; we remark that analogous properties for $F_p^o$ in \eqref{eq:p-norm-mls} do no exist. 
\begin{proposition} \label{prop:variance}
Given $\varepsilon$ and the vector of optimal load sheds $\bm z^*(F_{\varepsilon}^c)$ in \eqref{eq:fair-mls} with sample mean and variance $\mu_{\varepsilon}$ and $\sigma_{\varepsilon}$, respectively, then we have the following inequality
\begin{gather*}
 \sigma_{\varepsilon}^2 \leqslant \left(\frac{n}{\left(1-\varepsilon + \varepsilon\sqrt{n} \right)^2} -1\right)\mu_{\varepsilon}^2
\end{gather*}
where, $n$ is the total number of loads (dimension of the vector $\bm z^*(F_{\varepsilon}^c)$).
\end{proposition}
\begin{proof} 
The proof follows by expanding $\sigma_{\varepsilon}^2$ using Definition \ref{def:eps-fairness} in combination with the sample first and second moments. 
\end{proof}
Proposition \ref{prop:variance} shows that increasing $\varepsilon$ reduces the upper bound on the coefficient of variation ($\sigma_{\varepsilon}^2/\mu_{\varepsilon}^2$), thereby making the solution ``fair-er''. Next, we discuss trends in the feasibility and optimality of Problem \eqref{eq:fair-mls}.
\begin{proposition} \label{prop:eps-max}
Given $\bar\varepsilon \in [0, 1]$, if the problem $F_{\bar \varepsilon}^c$ is infeasible, then $F_{\varepsilon}^c$ is infeasible for any $\varepsilon \in [\bar \varepsilon, 1]$.
\end{proposition}
\begin{proof}
 To prove this statement, we let 
 \begin{gather*}
 \mathcal Y_{\mathrm{SOC}}^{\varepsilon} \triangleq \left\{ \bm d \in \mathcal X_{\mathrm{DC}} : \left(1-\varepsilon + \varepsilon\sqrt{|\mathcal D|} \right) \cdot \| \bm d \|_2 \leqslant \| \bm d \|_1 \right\}. 
 \end{gather*}
 The proof now follows from the observation that for any $\varepsilon_1, \varepsilon_2 \in [0, 1]$ such that $\varepsilon_1 > \varepsilon_2$, $\mathcal Y_{\mathrm{SOC}}^{\varepsilon_1} \subseteq \mathcal Y_{\mathrm{SOC}}^{\varepsilon_2}$.
\end{proof}
Intuitively, the above proposition states that if no operating point exists for the MLS problem to make the load sheds at least $\bar\varepsilon$-fair, then it also holds true for any $\varepsilon \in [\bar \varepsilon, 1]$. One consequence of the proposition is that there exits a $0\leqslant \varepsilon^{\max}\leqslant 1$ such that for any $\varepsilon \in (\varepsilon^{\max}, 1]$, $F_{\varepsilon}^c$ is infeasible, and $\varepsilon \in (0,\varepsilon^{\max}]$, $F_{\varepsilon}^c$ is feasible.
From here on, we shall refer to $[0, \varepsilon^{\max}]$ as the ``feasibility domain'' for $F_{\varepsilon}^c$. We now state another result that validates that any attempt to obtain a fair load shed allocation is achieved at the expense of increasing cost of operations. 
\begin{proposition} \label{prop:monotonicity}
The univariate function $\|\bm z^*(F_{\varepsilon}^c) \|_1$ is increasing in $\varepsilon$ for all values in the feasibility domain of $F_{\varepsilon}^c$. 
\end{proposition}
 \begin{proof}
 The proof follows from the following observations: (i) the feasible set of solutions of $F_{\varepsilon}^c$ becomes smaller as $\varepsilon$ increases (see proposition \ref{prop:eps-max}) and (ii) $\|\bm z^*(F_{\varepsilon}^c)\|_1$ is the total minimum load shed obtained by solving the optimization problem $F_{\varepsilon}^c$.
 \end{proof}
 \begin{corollary} \label{eq:corr-pof}
 $\mathrm{POF}(F_{\varepsilon}^c)$ is an increasing function of $\varepsilon$ for all values in the feasibility domain of $F_{\varepsilon}^c$.
 \end{corollary}
 The above corollary states that the relative increase in cost of operations measured in terms of the total load shed in the system increases as we strive to make the individual load sheds to all the customers more fair. 

\section{Case Studies} \label{sec:results}
In this section, we focus on extensive case studies that study the two methods of incorporating fairness in the MLS problem formulation in Sections \ref{sec:q1} and \ref{sec:q2}. The case studies are designed to compare and contrast the two methods in terms of (i) the distribution of load sheds that each method results in, (ii) the fairness of the load sheds across all customers measured using the indices presented in Sec. \ref{subsec:fairness-indices}, and finally, (iii) examining the cost-fairness trade-off of all the formulations using the price of fairness introduced in Sec. \ref{subsec:pof}. 

\subsection{Network data and computational platform} \label{subsec:data}
All the case studies are performed on the Power Grid Library's \cite{pglib} IEEE 14-bus test system with 10 loads. For the 14-bus network random damage scenarios were generated, with each scenario containing 5 damaged lines that caused non-zero total load shed for the MLS problem in \eqref{eq:mls}. This resulted in a total of 9765 scenarios on which the MLS problem and its fair variants were tested. 

The formulations presented in this paper were implemented in the Julia programming language \cite{bezanson2017julia} using the JuMP \cite{dunning2017jump} package. All linear and second-order cone optimization problems were solved using CPLEX \cite{cplex}, a commercial mixed-integer linear and convex quadratic programming solver where as all the nonlinear convex optimization problems were solved using IPOPT \cite{wachter2006implementation}. The source-code for the implementation, the case-studies, and the plots are all open-sourced and made available at \url{https://github.com/kaarthiksundar/FairlyExtreme}. As for the computational platform, all experiments were run on a Intel Haswell 2.6 GHz, 62 GB, 20-core machine at Los Alamos National Laboratory.

\subsection{Comparison with weighted version of MLS} \label{subsec:weighted_vs_unweighted}
As mentioned in the introduction, one way that existing literature enforces fairness and equity is by the use of weights in the objective function. Here we show that this approach is not effective in enforcing fairness, rather it is only good at enforcing priorities. To this end, Fig. \ref{fig:weighted-vs-unweighted-mls} shows the box plot of the individual load sheds for the 9765 scenarios. For the weighted version of the MLS, a weight of 2 units was used for the loads with ids 1 and 2, respectively and a weight of 1 unit for the rest. The weights were chosen this way to reduce the number the outliers in the load shed for these two loads for the MLS problem in Fig. \ref{fig:weighted-vs-unweighted-mls}. It is clear from the figure that these weights helped reduce the outliers in loads 1 and 2, but it did so at the cost of loads 3 and 6, i.e., the load shed at the loads with higher weight was moved to loads with lower weights. 

\begin{figure}[htbp]
 \centering
 \includegraphics[scale=0.7]{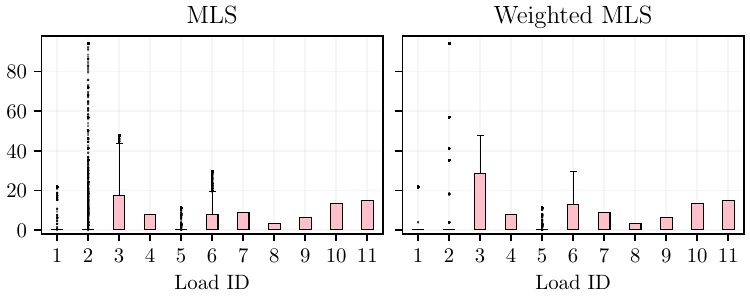}
 \caption{Box plots of the load sheds for the MLS problem and its weighted version. Here the unweighted version of the MLS is entitled ``MLS''. Observe that using weights moves the load shed from the loads with the highest priority (IDs 7 and 10) to the ones with low priority (IDs 4 and 8).}
 \label{fig:weighted-vs-unweighted-mls}
\end{figure}
Furthermore, Table \ref{tab:weighted_vs_unweighted} shows the statistics of the fairness indices obtained for the two versions of the MLS. It is clear that the distribution of load sheds across customers has not improved from a fairness stand-point when moving from an unweighted to a weighted MLS problem. We address this issue in subsequently in Section \ref{subsec:weights-fairness} and show that fairness and load prioritization can indeed be jointly optimized, using our $\varepsilon$-fairness approach, introduced in Section \ref{sec:q2}.
\begin{table}[htbp]
 \centering
 \caption{Statistics of fairness indices for the MLS and its weighted version. A fair solution will have a greater value of Gini index and a lower value of Jain et. al index.}
 \footnotesize
 \begin{tabular}{ccccc}
 \toprule
 \multirow{2}{*}{Statistic} & \multicolumn{2}{c}{Gini Index} & \multicolumn{2}{c}{Jain et. al Index} \\ 
 \cmidrule{2-5} 
 & unweighted & weighted & unweighted & weighted \\
 \midrule 
 mean & 0.86 & 0.85 & 0.23 & 0.24 \\ 
 max. & 1.00 & 1.00 & 0.65 & 0.65 \\ 
 min. & 0.43 & 0.43 & 0.09 & 0.09 \\
 \bottomrule
 \end{tabular}
 \label{tab:weighted_vs_unweighted}
\end{table}

\begin{figure}[htbp]
 \centering
 \includegraphics[scale=0.7]{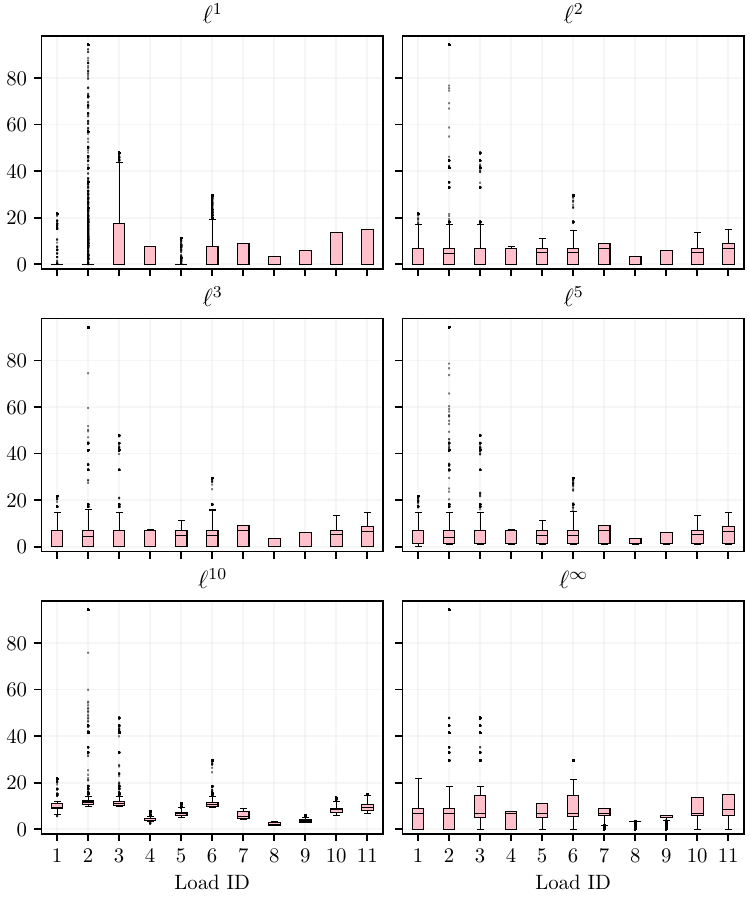}
 \caption{Box plot of load sheds for every load over all the scenarios for $F^o_p$ with varying values of $p$. Here, the $y$-axis is the load shed in MW.}
 \label{fig:ls-q1}
\end{figure}

\subsection{Fairness using $\ell^p$ norms} \label{subsec:obj-fairness}
Here, we examine the impact of using $\ell^p$ norms in the objective function of the MLS on the fairness of the solutions produced by the same. To this end, Fig. \ref{fig:ls-q1} shows the box plot of the load sheds for $F_p^o$ for $p \in \{1, 2, 3, 5, 10, \infty\}$. Observe that as we move from $p = 1 \rightarrow \infty$, the distribution of load sheds across all scenarios tries to converge to the same distribution for every load. Nevertheless, the other constraints in $\mathcal X_{\mathrm{DC}}$ does not permit them to be equal to one another. 

To obtain a more quantitative view of how fair the different solutions are, we present Fig. \ref{fig:indexes-q1} which shows the box plot of the two fairness indices for the solutions of $F_p^o$. We remark that neither fairness index is a monotonic function of the $p$-norm used in the objective function, i.e, as we increase the value of $p$, it is not necessary for the loads sheds to become fairer. Fig. \ref{fig:pof-q1} shows a box plot of the POF values for different values of $p$. Out of the total of 9765 scenarios, monotonicity of POF with $p$ does not hold for 5166 scenarios. This makes the problem of selecting $p$ to obtain the most fair load shed (measured in terms of fairness indices) difficult to solve. Nevertheless, one can always choose a few values of $p$ heuristically, obtain the optimal solutions for $\ell^p$ norms, and analyse them both in terms of fairness indices and POF values, to determine an appropriate $p$.
\begin{figure}[htbp]
 \centering
 \includegraphics[scale=0.7]{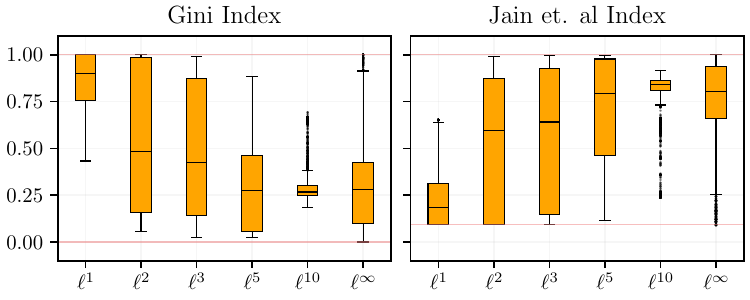}
 \caption{Box plot of the fairness indices for the optimal solutions of $F_p^o$. The minimum and maximum values the indices can take is shown with a red line.}
 \label{fig:indexes-q1}
\end{figure}
\begin{figure}[htbp]
 \centering
 \includegraphics[scale=0.7]{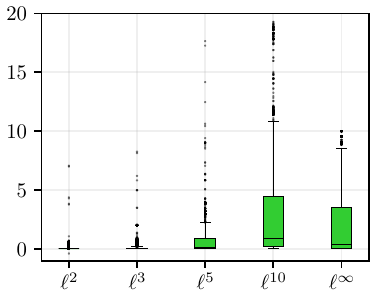}
 \caption{Box plot of the POF values for the optimal solutions of $F_p^o$.}
 \label{fig:pof-q1}
\end{figure}
The maximum of POF values for each $p$ is shown in Table \ref{tab:pof-max-q1}. This indicates that there are scenarios for which the operator has to pay a prohibitive cost to enforce fairness, bringing to question this approach's practicality for arbitrary $\ell^p$ norms. But, if an operator is restricted to using this approach, Fig. \ref{fig:indexes-q1} and \ref{fig:indexes-q2}, and results in Table \ref{tab:pof-max-q1} indicate that selecting $p$ from the set $\{2, 3, \infty\}$ provides a good balance of cost increase and fairness.
\begin{table}[htbp]
 \centering
 \caption{Maximum POF values for different $p$-norms.}
 \footnotesize
 \begin{tabular}{cccccc}
 \toprule 
 $p$ & 2 & 3 & 5 & 10 & $\infty$ \\ 
 \midrule
 POF & 4.38 & 101.84 & 1420.98 & 7280.17 & 10.0 \\
 \bottomrule
 \end{tabular}
 \label{tab:pof-max-q1}
\end{table}

\subsection{Fairness using a SOC constraint} \label{subsec:cons-fairness}
We now present results to show the effectiveness of the SOC constraint for $\varepsilon$-fairness in obtaining fair solutions to the MLS. The SOC constraint is parameterized by $\varepsilon$ that can take values in $[0, 1]$. To show the impact of $\varepsilon$, we vary it from 0 to 1 in steps of 0.1. We start by presenting the box plot of load sheds for all scenarios in Fig. \ref{fig:ls-q2} to provide a qualitative view of fairness. 

\begin{figure}[htbp]
 \centering
 \includegraphics[scale=0.7]{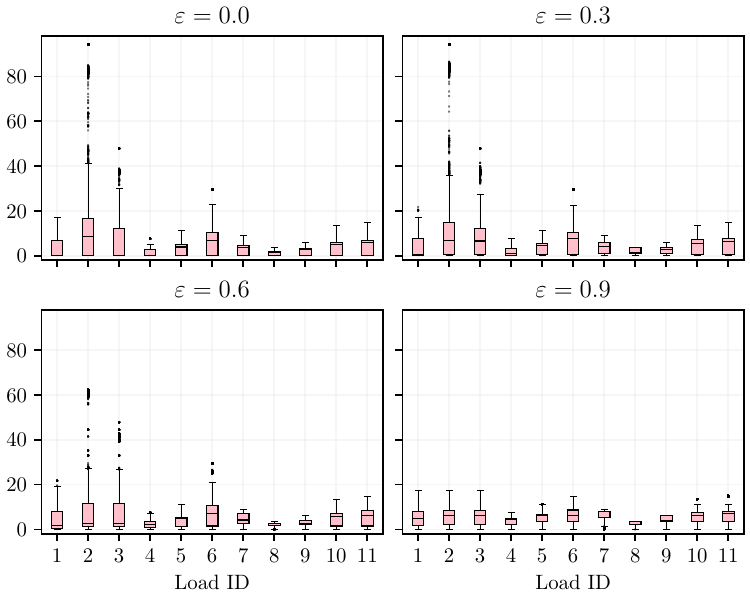}
 \caption{Box plot of load sheds for $F^c_{\varepsilon}$. In this case, it is known that as we increase $\varepsilon$, the solution of $F^c_{\varepsilon}$, if it exists, become more fair.}
 \label{fig:ls-q2}
\end{figure}

The fairness indices for varying values of $\varepsilon$ is shown in Fig. \ref{fig:indexes-q2}. Unlike the fairness indices in Fig. \ref{fig:indexes-q1}, here the value of Gini (Jain et. al) index monotonically decreases (increases) with increasing $\varepsilon$. This is true because the parameter $\varepsilon$ is a direct measure of fairness in the SOC constraint in \eqref{eq:soc-e-fair}. We note that for both the indices, the box plot is shown only for the scenarios that are feasible for all $\varepsilon \in [0, 0.9]$. When the value of $\varepsilon$ was set to 1, only 634 instances were feasible i.e., for these instances $F_{\varepsilon}^c$ is feasible for all values of $\varepsilon \in [0, 1]$.
\begin{figure}[htbp]
 \centering
 \includegraphics[scale=0.7]{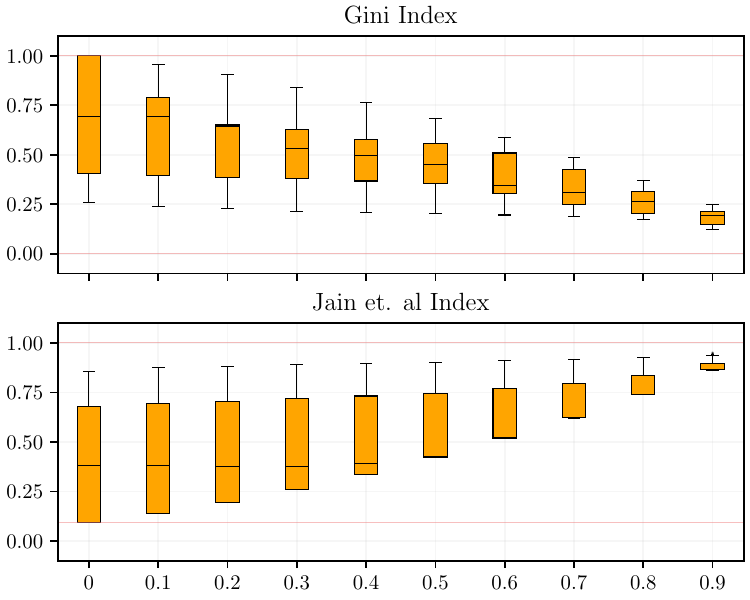}
 \caption{Fairness indices for $F_{\varepsilon}^c$ with $\varepsilon$ varying from 0 to 0.9 in steps of 0.1. The scenarios that are feasible for all values of $\varepsilon$ in the plot are used for the box plot which equals 6610. The minimum and maximum values the indices can take is shown with a red line.}
 \label{fig:indexes-q2}
\end{figure}

Table \ref{tab:infeasible-q2} shows the number of scenarios for which $F^c_{\varepsilon}$ is infeasible for different values of $\varepsilon$. The number of infeasible scenarios increases with $\varepsilon$, validating the correctness of proposition \ref{prop:eps-max} empirically. Note that if for a given $\varepsilon \in [0, 1]$, no $\varepsilon$-fair solution exists, the value of $\varepsilon$ has to be reduced to ensure $\varepsilon < \varepsilon^{\max}$ and the value of $\varepsilon^{\max}$ can be computed systematically using bisection on $[0, 1]$.

\begin{table}[htbp]
 \centering
 \caption{Number of scenarios $F_{\varepsilon}^c$ for different values of $\varepsilon$.}
 \footnotesize
 \begin{tabular}{cc}
 \toprule
 $\varepsilon$ & \# infeasible scenarios \\
 \midrule 
 $[0.0, 0.5]$ & 0/9765 \\ 
 $[0.6, 0.7]$ & 813/9765 \\ 
 $0.8$ & 1443/9765 \\ 
 $0.9$ & 1715/9765 \\ 
 % $1.0$ & 9131/9765 \\
 \bottomrule
 \end{tabular}
 \label{tab:infeasible-q2}
\end{table}
\begin{figure}[htbp]
 \centering
 \includegraphics[scale=0.7]{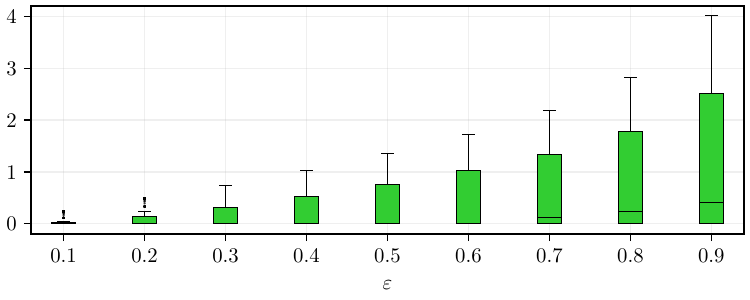}
 \caption{Box plot of the POF values for the optimal solutions of $F_{\varepsilon}^c$ for different values of $\varepsilon$.}
 \label{fig:pof-q2}
\end{figure}
\begin{table}[htbp]
 \centering
 \caption{Maximum POF values for different $\varepsilon$ values}
 \footnotesize
 \begin{tabular}{cccccc}
 \toprule 
 $\varepsilon$ & 0.1 & 0.3 & 0.5 & 0.7 & 0.9 \\ 
 \midrule
 POF & 0.24 & 0.74 & 1.42 & 2.34 & 4.02\\
 \bottomrule
 \end{tabular}
 \label{tab:pof-max-q2}
\end{table}

We conclude this study by showing the monotonicity of the POF with $\varepsilon$ through the box plot in \ref{fig:pof-q2}. Again, similar to Fig. \ref{fig:indexes-q2} only the 6610 scenarios that were feasible for all values of $\varepsilon \in [0, 0.9]$ are used for the box plot. The maximum of the POF values for different $\varepsilon$ values is shown in table \ref{tab:pof-max-q2}. These maximum values are one to three orders of magnitude lower than the those in table \ref{tab:pof-max-q1} because unlike the approach to obtain fairness by minimizing $\ell^p$ norm, this approach aims to find the minimum cost solution such that the solution is at least $\varepsilon$-fair.

\subsection{Incorporating priorities and fairness simultaneously} \label{subsec:weights-fairness}
The approach of adding a fairness constraint into MLS does not forbid the use of weights to model priorities and forthcoming results analyse the impact of using weights to model priorities for different loads and enforcing fairness as a constraint in the weighted MLS problem in \eqref{eq:weighted-mls}. To this end, we use the same weights in Sec. \ref{subsec:weighted_vs_unweighted} for the weighted MLS problem and allow $\varepsilon$ in the SOC constraint for fairness to take any value in the set $\{0.2, 0.4, 0.6, 0.8\}$. The first set of results in Fig. \ref{fig:indexes-weighted} show the box plot of the Gini and Jain et. al indices for the load sheds in the different scenarios when priorities and fairness are enforced simultaneously. It is clear from the plots that adding fairness constraint to the weighted MLS problem clearly makes the load sheds fair. 
\begin{figure}[htbp]
 \centering
 \includegraphics[scale=0.7]{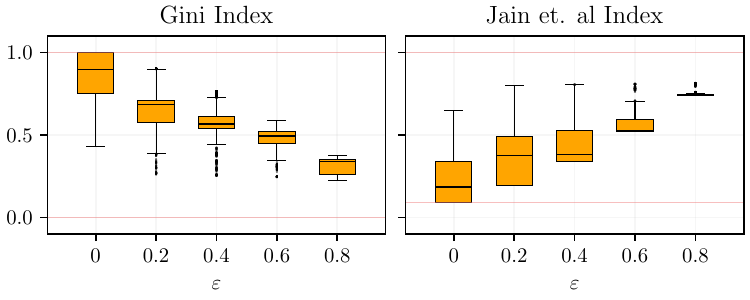}
 \caption{Fairness indices for the solutions when weights are used to prioritize loads and fairness is enforced as a constraint. }
 \label{fig:indexes-weighted}
\end{figure}

To see that priorities are also being taken into account when enforcing $\varepsilon$-fairness, we present Table \ref{tab:priority-metrics} that compares the average sum of load shed in the high priority loads (loads 1 and 2) relative to the total load shed, between the unweighted and weighted version of the MLS problem. It is clear from the table that as we move from the unweighted to the weighted setting, the total load shed contributed by the high priority loads decreases (in our case by a factor of $3\sim 4$ for $\varepsilon > 0$). This shows that the novel technique to enforce fairness as a constraint can work in tandem with prioritized load shed as well. 

\begin{table}[htbp]
 \centering
 \caption{Mean of the sum of load sheds in the high priority loads relative to the total load shed in the system (in \%) for unweighted and weighted MLS problems while enforcing $\varepsilon$-fairness.}
 \footnotesize
 \begin{tabular}{ccc}
 \toprule
 $\varepsilon$ & MLS & Weighted MLS \\
 \midrule 
 0.0 & 15.77 & 9.79\\ 
 0.2 & 22.60 & 7.41 \\ 
 0.4 & 21.32 & 6.00 \\ 
 0.6 & 17.91 & 2.45 \\ 
 0.8 & 16.96 & 4.71 \\ 
 \bottomrule
 \end{tabular}
 \label{tab:priority-metrics}
\end{table}

\section{Conclusion \& Way Forward} \label{sec:conclusion}
This paper presents two general ways to incorporate fairness into existing optimization problems in the power systems literature that focus on minimizing cost of operations or planning. One way is by changing the objective function which has been used in the literature before in the context of other applications. The other technique is through the introduction of a novel notion of fairness referred to as $\varepsilon$-fairness that can directly be incorporated into the optimization problem through a single second-order cone constraint. Through extensive case studies on a minimum load shedding problem that is used to minimize outages on a damaged power system, it is shown that incorporating fairness as a constraint is flexible to study the cost-fairness trade-off that inherently exists when modeling fairness and conforms to our intuitive understanding of this trade-off through mathematical properties like monotonicty. Future work would focus on showing the effectiveness of these approaches in obtaining equitable expansion plans for the power grid.

\bibliography{fairness-refs}

\bibliographystyle{IEEEtran}
\end{document}